\documentclass{article}
\usepackage{amsmath,amsthm,amssymb,amscd}

\newcommand{\ga}{\alpha}
\newcommand{\gb}{\beta}
\renewcommand{\gg}{\gamma}
\newcommand{\gd}{\delta}
\newcommand{\gw}{\omega}

\newcommand{\gs}{\sigma}

\newcommand{\liff}{\leftrightarrow}

\newcommand{\cantor}{2^\gw}
\newcommand{\baire}{\gw^\gw}

\newcommand{\gwtree}{\gw^{<\gw}}

\newcommand{\dom}{\mathrm{dom}}

\newcommand{\power}{\mathcal{P}}

\newcommand{\hvd}{\mathrm{HOD}}
\newcommand{\vd}{\mathrm{OD}}

\newcommand{\coll}{\mathrm{Coll}}
\newcommand{\eval}{\mathrm{ev}}
\newcommand{\forkindep}[1][]{%
  \mathrel{
    \mathop{
      \vcenter{
        \hbox{\oalign{\noalign{\kern-.3ex}\hfil$\vert$\hfil\cr
              \noalign{\kern-.7ex}
              $\smile$\cr\noalign{\kern-.3ex}}}
      }
    }\displaylimits_{#1}
  }
}

\newtheorem{theorem}{Theorem}[section]

\newtheorem{claim}[theorem]{Claim}
\newtheorem{corollary}[theorem]{Corollary}
\newtheorem{fact}[theorem]{Fact}
\newtheorem{proposition}[theorem]{Proposition}

\theoremstyle{definition}
\newtheorem{definition}[theorem]{Definition}

\newtheorem{question}[theorem]{Question}
\title{Independence relations in the Solovay model I\footnote{2020 AMS subject classification 03E25, 03E30.}}

\author{
Jind{\v r}ich Zapletal\\
University of Florida\\
zapletal@ufl.edu}

\begin{document}
\maketitle

\begin{abstract}
I provide a novel axiomatization of the Solovay model using the concept of an independence relation with properties familiar from geometric model theory. This serves as a tool for forcing-free proofs of many classical results as well as a basis for a purely geometric treatment of the theory of balanced forcing.
\end{abstract}

\section{Introduction}

The Solovay model \cite{solovay:model} \cite[Chapter 26]{jech:newset} \cite[Chapter 8]{schindler:set} is one of the most thoroughly studied objects in set theory. It was developed by Solovay as a model of choiceless set theory in which all sets of reals are Lebesgue measurable. Its usual construction provides a way of analyzing its theory in great detail. However, this approach uses forcing in a way which makes proofs rather complex and repetitive at the same time.

In \cite{z:geometric}, Paul Larson and the present author developed a flexible theory of forcing extensions of the Solovay model called balanced forcing. The theory uses the standard approach to the Solovay model and results in rather repetitive and unwieldy proofs. The purpose of this series of papers is to provide a ``geometric" axiomatization of the Solovay model, and then develop a novel presentation of the theory of balanced forcing in the Solovay model in a similar spirit. The word ``geometric" here is to suggest a parallel with geometric model theory and the use of independence relations between models as in \cite{pillay:simple}. The axiomatization provided is in no sense complete, but it enables a novel and nearly forcing-free style of arguments. In this paper, I prove older results about the Solovay model such as measurability of all sets of reals (\cite{solovay:model}, Theorem~\ref{solovaytheorem}), the Ramsey property of subsets of $[\gw]^{\aleph_0}$ (\cite{mathias:happy}, Theorem~\ref{mathiastheorem}), nonexistence of MAD families (\cite{tornquist:mad, mathias:happy}, Theorem~\ref{tornquisttheorem}), nonexistence of uncountable sequences of Borel sets of bounded complexity (\cite{stern:borel}, Theorem~\ref{sterntheorem}) or a version of the Silver dichotomy (\cite{stern:borel}, Theorem~\ref{silvertheorem}) from this axiomatic basis, all with short, forcing-free and even descriptive set theory-free arguments. In the continuation, I provide a geometric presentation of balanced forcing which is methodologically much simpler than that of \cite{z:geometric}; in addition, it applies to much broader class of posets and does away with the needless complexity calculations which plague most of that book.

As for the structure of this paper, in Section~\ref{definabilitysection} I discuss the language of definability. Section~\ref{independencesection} isolates the properties of an independence relation needed for the analysis of the Solovay model, remarkably similar to forking independence. In Section~\ref{inaccessibilitysection}, I isolate the inaccessibility axiom and show that in conjunction with the independence axiom, it has many mathematical consequences; in particular, it implies that no MAD families exist. In Section~\ref{bairesection}, I consider the axiom stating that every set of reals has the Baire property. In the presence of the independence relation, this is a strong statement, implying in particular that all sets of reals are Lebesgue measurable, and that the independence relation is essentially mutual genericity. In Section~\ref{ghasection}, I consider a Gandy--Harrington type of a game and use the associated axiom to prove the Silver dichotomy. This is one axiom which is not geometric in nature, but it provides quick and clutter-free proofs of several dichotomy theorems. Finally in Section~\ref{modelsection}, I discuss various slightly different types of the Solovay model and show that the (usual version of the) Solovay model satisfies all axioms on the list.

The terminology and notation used in the paper follows the set theoretic standard of \cite{jech:newset}. During the work on this paper, the author was partially supported by NSF grant DMS 2348371.

\section{The definability axiom}
\label{definabilitysection}

The basic axiomatic framework for this paper is ZF+DC, where DC stands for the axiom of dependent choices. The first additional axiom, without which the independence axiomatization has no teeth, is the definability axiom.

\begin{definition}
The \emph{definability axiom} is the statement that every set is definable from a set of ordinals.
\end{definition}

\noindent By ``definable", I mean definable in the universe of sets using the first order language with the membership relation.

 I use this opportunity to set up some notation. If $x$ is a set of ordinals then the symbol $\vd_x$ denotes the class of all sets which are definable from $x$ and an additional finite tuple of ordinals. If $Z$ is a Polish space, by $Z\in\vd_x$ I will always mean to say that $\vd_x$ contains the underlying set as well as some countable basis of the topology of $Z$. The symbol $\hvd_x$ denotes the transitive part of $\vd_x$; that is, the class of all sets $z$ such that the transitive closure of $z$ is a subset of $\vd_x$. The following are standard facts.

\begin{fact}
Let $x$ be a set of ordinals.

\begin{enumerate}
\item there is a class well-ordering on $\vd_x$ definable from $x$;
\item $\hvd_x$ is a model of ZFC.
\end{enumerate}
\end{fact}

\noindent I will abuse the above notation a bit. I will identify sets of ordinals and finite tuples of sets of ordinals. Thus, if $x$ and $y$ are sets of ordinals, then $\vd_{xy}$ stands for the class of all sets definable from $x$, $y$, and an additional finite tuple of ordinals. Also, if $y\subset\vd_x$, I will write $\vd_{xy}$ as if $y$ is a set of ordinals: $y$ can be identified with the set of the indexes of its elements in the canonical definable well-order of $\vd_x$. Another similar shorthand appears when $Z$ is a Hausdorff space with a well-orderable basis $\mathcal{B}$, and $z\in Z$ is a point. If $x$ is a set of ordinals such that $Z, \mathcal{B}\in\vd_x$, then $\mathcal{B}\subset\vd_x$ will hold as well, by Corollary~\ref{wocorollary} below. In such a situation, the point $z$ can be identified with the set $\{O\in\mathcal{B}\colon z\in O\}\subset\vd_x$, and in this sense I will use the notation $\vd_{xz}$.

The following easy proposition will be used to compactify notation in many arguments of this paper.

\begin{proposition}
Let $x$ be a set of ordinals and $A$ be a set. The following are equivalent:

\begin{enumerate}
\item $A\in\vd_x$;
\item there is a function $F$ definable from a finite tuple of ordinals such that $x\in\dom(F)$ and $F(x)=A$.
\end{enumerate}
\end{proposition}

\begin{proof}
The (2)$\to$(1) implication does not need an argument. For the opposite direction, suppose that $\phi(v, x, \bar\ga)$ is a formula with one free variable displayed and additional parameters $x$ and a finite tuple $\bar\ga$ of ordinals, such that $\phi$ has only one solution and the solution is $A$. Define the (class) function $F$ by $F(y)=B$ if the formula $\phi(v, y, \bar\ga)$ has only one solution and the solution is $B$, and $F(y)=0$ if the formula $\phi(v, y, \bar\ga)$ has no solutions or more than one solutions. It is clear that $F$ works as desired.
\end{proof}

\section{The independence axiom}
\label{independencesection}

The most important contribution of this paper is the axiomatization of set theoretic independence and its uses. It will turn out much later that the independence relation is necessarily essentially equal to mutual genericity (Proposition~\ref{kimpillayproposition}), but one can use the independence relation successfully in arguments axiomatically without ever considering any forcing language.

\begin{definition}
A (set theoretic) \emph{independence relation} is a ternary class relation $\forkindep$ on sets of ordinals satisfying the following properties:

\begin{enumerate}
	\item (nontriviality) $x\forkindep[x]x$ holds for all $x$;
\item (symmetry) $y_0\forkindep[x]y_1$ implies $y_1\forkindep[x]y_0$;
\item (monotonicity) if $x, y_0, y_1, y'_0$ are sets of ordinals and $y'_0\in\vd_{xy_0}$ then $y_0\forkindep[x]y_1$ implies $y'_0\forkindep[x]y_1$;
\item (transitivity) $y_0\forkindep[x]y_1$ and $z\forkindep[xy_0]y_1$ implies $y_0z\forkindep[x]y_1$;
\item (extension) if $x$ is a set of ordinals and $A\in\vd_x$ is a nonempty set of sets of ordinals, then there are $y_0, y_1\in A$ such that $y_0\forkindep[x]y_1$ holds;
\item (product) if $y_0\forkindep[x]y_1$ holds, $A\in\vd_x$, and $\langle y_0, y_1\rangle\in A$ holds, then there are $\vd_x$-sets $B_0, B_1$ such that $y_0\in B_0$, $y_1\in B_1$, and for every $y'_0\in B_0$ and $y'_1\in B_1$, $y'_0\forkindep[x]y'_1$ implies $\langle y'_0, y'_1\rangle\in A$. 
\end{enumerate}
\end{definition}

\begin{definition}
The \emph{independence axiom} asserts that there is a class independence relation $\forkindep$ definable from no parameters.
\end{definition}

\noindent The properties stated are motivated by their counterparts in model theory \cite[Section 6.3]{marker:book}, \cite{pillay:simple}. The extension item will get a more familiar form in the next proposition. It may be satisfied in a trivial way: if $A$ has a $\vd_x$ element $y$ (for example, if $A$ is a singleton), then $y_0=y_1=y$ is going to work. The product item says that the theory of an independent pair indexed by $x$ and ordinals depend continuously on the theories of its left and right coordinates.

In this section, I will derive a couple of basic consequences of the independence axiom which will make it easier to use it in applications. All statements are proved in ZF+DC+definability+independence. In the statement of the extension and product properties, the set $A$ can be replaced by a \emph{class} definable from $x$ and a finite tuple of ordinals without changing the truth value of the statement. The extension property is typically used via its following apparent strengthening:

\begin{proposition}
\label{extensionproposition}
Let $x, y_0$ be sets of ordinals.

\begin{enumerate}
\item If $A\in\vd_x$ is a nonempty set of sets of ordinals, then there is $y_1\in A$ such that $y_0\forkindep[x]y_1$ holds;
\item $x\forkindep[x]y_0$ holds.
\end{enumerate}
\end{proposition}

\begin{proof}
Suppose towards a contradiction that (1) fails for $x, y_0$ and $A$; there must be an ordinal $\ga$ such that the set $B=\{y\subset\ga\colon\forall z\in A\ \lnot y\forkindep[x]z\}$ is nonempty. The set $B\times A$ is $\vd_x$ and nonempty, so by the extension property applied to it there are pairs $\langle y_0, z_0\rangle$ and $\langle y_1, z_1\rangle$ in $B\times A$ such that $y_0z_0\forkindep[x]y_1z_1$. The monotonicity property applied twice then shows that $y_0\forkindep[x]z_1$, contradicting the choice of $y_0$.

Suppose towards a contradiction that (2) fails. Then there is an ordinal $\ga$ such that the set $B=\{y\subset\ga\colon \lnot x\forkindep[x]y\}$ is nonempty. By (1), there must be a set $y\in B$ such that $x\forkindep[x]y$ holds, contradicting the definition of the set $B$.
\end{proof}

\noindent Now I turn to the transitivity property. The following proposition shows among other things that the transitivity holds in its equivalence form more familiar from model theory. 

\begin{proposition}
	\label{transitivityproposition}
	Let $x, y_0, y_1, z$ be sets of ordinals such that $y_0\forkindep[x]y_1$.
	
	\begin{enumerate}
		\item $z\forkindep[xy_0]y_1$ is equivalent with $y_0z\forkindep[x]y_1$;
		\item if $z\in\vd_{xy_0}$ then $y_0\forkindep[xz]y_1$;
		\item if $z\forkindep[x]y_0y_1$ then $y_0z\forkindep[x]y_1$.
	\end{enumerate}
\end{proposition}

\begin{proof}
	For (1), the left-to-right implication is the transitivity property. To prove the right-to-left implication, suppose that it fails. By the product property, there must be $\vd_x$-sets $B_0$ and $B_1$ such that $\langle y_0, z\rangle\in B_0$, $y_1\in B_1$, and for all $\langle y'_0, z'\rangle\in B_0$ and $y'_1\in B_1$, if $y'_0z'\forkindep[x]y'_1$ then $z'\forkindep[xy'_0]y'_1$ fails. Let $C$ be the projection of $B_0$ into the first coordinate. Use the extension property to find $y'_0\in C$ and $y'_1\in B_1$ such that $y'_0\forkindep[x]y'_1$. Use the extension property again to find $z'\in (B_0)_{y'_0}$ such that $z'\forkindep[xy'_0]y'_1$. The transitivity property implies that $y'_0z'\forkindep[x]y'_1$. Since $z'\forkindep[xy'_0]y'_1$ holds, this is a contradiction with the choice of the sets $B_0$ and $B_1$.
	
	For (2), first observe that $y_0z\forkindep[x]y_1$ holds by monotonicity, and then apply (1) with the roles of $y_0$ and $z$ interchanged. For (3), apply (1) twice, first pushing $y_0$ into the subscript of the $\forkindep$ sign and then to the $z$-side of it.
\end{proof}

\begin{corollary}
	If $x$ is a set of ordinals, $\bar y$ and $\bar z$ are $n$-tuples of sets of ordinals for some $n\in\gw$, and $A\in\vd_x$ is a nonempty set of sets of ordinals, then there is $u\in A$ such that for all $i\in n$, $u\forkindep[x\bar y(i)]\bar z(i)$.
\end{corollary}

\begin{proof}
	Use the extension property to find $u\in A$ such that $u\forkindep[x]\bar y\bar z$. Then, for each $i\in n$, $u\forkindep[x]\bar y(i)\bar z(i)$ holds by monotonicity, and $u\forkindep[x\bar y(i)]\bar z(i)$ follows from Proposition~\ref{transitivityproposition}(2).
\end{proof}

\noindent The transitivity property makes it possible to talk about independent tuples and produce them in a staightforward inductive manner.

\begin{definition}
	Let $x$ be a set of ordinals and $\bar y$ be a finite tuple of sets of ordinals. Say that $\bar y$ is \emph{independent over} $x$ if either $\bar y$ has at most one element, or for any two non-overlapping sub-tuples $\bar u_0, \bar u_1$ of $\bar y$, $\bar u_0\forkindep[x]\bar u_1$ holds.
\end{definition}

\begin{proposition}
	\label{tupleproposition}
	Let $x$ be a set of ordinals and $\bar y_0, \bar y_1$ be finite tuples of ordinals. The following are equivalent:
	
	\begin{enumerate}
		\item $\bar y_0^\smallfrown\bar y_1$ is independent over $x$;
		\item $\bar y_0, \bar y_1$ are each independent over $x$ and $\bar y_0\forkindep[x]\bar y_1$.
	\end{enumerate}
\end{proposition}

\begin{proof}
	(1)$\to$(2) follows immediately from the definitions. The opposite implication is proved by induction on the length of $\bar y_0$, moving the elements of $\bar y_0$ and $\bar y_1$ to the left or right of the $\forkindep$ sign using Proposition~\ref{transitivityproposition}(3).
\end{proof}

\begin{corollary}
	\label{tuplecorollary}
	Let $x$ be a set of ordinals, $n\in\gw$ be a number, and $\bar B$ be an $n$-tuple of nonempty $\vd_x$-sets of sets of ordinals. There is an $n$-tuple $\bar y$ independent over $x$ such that its entries belong to the corresponding entries on $\bar B$.
\end{corollary}

\begin{proof}
	By induction on $n$ for each $x$ simultaneously. The induction step is performed with the help of Propositions~\ref{tupleproposition} and~\ref{extensionproposition}.
\end{proof}

\begin{corollary}
	\label{duplecorollary}
	Let $x$ be a set of ordinals, $n\in\gw$ be a number, $\bar y$ be an $n$-tuple of sets of ordinals independent over $x$, and $A$ be a $\vd_x$ set containing $\bar y$. Then there is an $n$-tuple $\bar B$ of $\vd_x$-sets such that $\forall i\in n\ \bar y(i)\in\bar B(i)$, and such that for every $n$-tuple $\bar y'$ independent over $x$ such that $\forall i\in n\ \bar y'(i)\in\bar B$, $\bar y'\in A$ holds.
\end{corollary}

\begin{proof}
	By induction on $n$. The case of $n=1$ is trivial. For the induction step, suppose that $\bar y$ is an independent $n+1$-tuple in some $\vd_x$-set $A$. Use the product property to find sets $C_0$ and $C_1$ such that $\bar y_0\restriction n\in C_0$, $\bar y_0(n)\in C_1$, $C_0$ consists of independent $n$-tuples, and for every $\bar u\in C_0$ and $v\in C_1$, if $\bar u\forkindep[x] v$ then $\bar u^\smallfrown v\in A$. Use the induction hypothesis on $C_0$ to find the related $n$-tuple $\bar B\restriction n$, and finally obtain $\bar B$ by appending $C_1$ to $\bar B\restriction n$.
\end{proof}

\noindent The product property has a number of immediate consequences which are going to be used repeatedly throughout this paper.

\begin{proposition}
\label{separationproposition}
If $y_0\forkindep[x]y_1$ and $A_0\in\vd_{xy_0}$ and $A_1\in\vd_{xy_1}$ are disjoint subsets of $\vd_x$ then there is a set $D\in\vd_x$ such that $A_0\subseteq D$ and $A_1\cap D=0$.
\end{proposition}

\begin{proof}
Let $F_0$ and $F_1$ be $\vd$ functions such that $F_0(x, y_0)=A_0$ and $F_1(x, y_1)=A_1$. Using the product property, let $B_0$ and $B_1$ be $\vd_x$ sets such that $y_0\in B_0$, $y_1\in B_1$, and for every $y'_0\in B_0$ $F_0(x, y_0)\subset\vd_x$ holds, for every $y'_1\in B_1$ $F_1(x, y'_1)\in\vd_x$ holds, and $y'_0\forkindep[x]y'_1$ implies $F_0(x, y'_0)\cap F_1(x, y'_1)=0$. Let $D=\bigcup\{F_0(x, y'_0)\colon y'_0\in B_0\}$. This is clearly a $\vd_x$ set which contains $A_0$ as a subset. Suppose towards a contradiction that $A_1\cap D$ is nonempty, containing some element $C$. Consider the nonempty $\vd_x$ set $B'_1=\{y'_1\in B_1\colon C\in F_1(x, y'_1)\}$. Let $y'_0\in B_0$ be an element such that $C\in F_0(x, y'_0)$. Use Proposition~\ref{extensionproposition} to find $y'_1\in B'_1$ such that $y'_0\forkindep[x]y'_1$. Then $C\in F_0(x, y'_0)\cap F_1(x, y'_1)$, contradicting the choice of the sets $B_0, B_1$.
\end{proof}

\begin{proposition}
\label{intersectionproposition}
$y_0\forkindep[x]y_1$ implies $\vd_{xy_0}\cap\vd_{xy_1}=\vd_x$.
\end{proposition}

\begin{proof}
Suppose that $A\in\vd_{xy_0}\cap\vd_{xy_1}$ is a set; I must show that $A\in\vd_x$ holds. Let $F_0, F_1$ be $\vd$ functions such that $A=F_0(x, y_0)=F_1(x, y_1)$. Use the product property to find $\vd_x$-sets $B_0, B_1$ such that $y_0\in B_0$, $y_1\in B_1$, and for all $y'_0\in B_0$ and $y'_1\in B_1$, $y'_0\forkindep[x]y'_1$ implies $F_0(x, y'_0)=F_1(x, y'_1)$. I will show that for all $y'_0\in B_0$, $F_0(x, y'_0)=A$; this certainly makes $A$ definable as the unique $F_0$ value on $B_0$ and proves the proposition.

Let $y'_0\in B_0$ be arbitrary. Use Proposition~\ref{extensionproposition} to find $y'_1\in B_1$ such that $y_0y'_0\forkindep[x] y'_1$. The monotonicity property implies that both $y_0\forkindep[x]y'_1$ and $y'_0\forkindep[x]y'_1$ hold. The choice of the sets $B_0, B_1$ then implies that $F_0(x, y_0)=F_1(x, y'_1)=F_0(x, y'_0)$ and the latter value has to be $A$ by the transitivity of equality.
\end{proof}

\begin{corollary}
\label{wocorollary}
If $x$ is a set of ordinals and $A\in\vd_x$ is a set, the following are equivalent:

\begin{enumerate}
\item $A$ is well-orderable;
\item $A\subset\vd_x$.
\end{enumerate}
\end{corollary}

\begin{proof}
(2) implies (1) as $\vd_x$ is well-orderable. For (1) implies (2), let $B=\{y\colon y$ is a set of ordinals such that $\vd_y$ contains a well-ordering of $A\}$. Note that if $y\in B$, then $A\subset\vd_y$. By the extension property, I can find sets $y_0, y_1\in B$ such that $y_0\forkindep[x]y_1$. By Proposition~\ref{intersectionproposition}, $\vd_{xy_0}\cap\vd_{xy_1}=\vd_{x}$. As $A\subset\vd_{xy_0}$ and $A\subset\vd_{xy_1}$ both hold, $A\subset\vd_x$ follows.
\end{proof}

\begin{corollary}
	\label{idealcorollary}
	Let $x$ be a set of ordinals, and $Z, I, A\in\vd_x$ be a Polish space, a set of closed subsets of $Z$ closed under subset, and a subset of $Z$ respectively. The following are equivalent:
	
	\begin{enumerate}
		\item $A$ is a subset of a union of a well-orderable subset of $I$;
		\item $A\subseteq \bigcup (I\cap \vd_x)$.
	\end{enumerate}
\end{corollary}

\noindent In particular, every $F_\gs$ $\vd_x$-set is a union of all its closed subsets in $\vd_x$, as follows from the corollary applied to $I=$the collection of all its closed subsets.

\begin{proof}
	(2) implies (1) as $\vd_x$ is well-orderable. For (1) implies (2), let $B$ be the union of all closed sets in $\vd_x\cap I$, and work to show that $A\subseteq B$.
	
	If not, let $y_0$ be any set of ordinals such that there is a point $z\in\vd_{xy_0}$ in $A\setminus B$. Consider the nonempty class $D=\{y\colon$ in $\vd_{xy}$ there is a well-orderable collection of sets in $I$ whose union covers $A\}$ definable from $x$ and an additional finite tuple of ordinal parameters. By Proposition~\ref{extensionproposition}, there is a point $y_1\in D$ such that $y_0\forkindep[x]y_1$ holds. By Corollary~\ref{wocorollary}, $A\subset\bigcup I\cap\vd_{xy_1}$ holds. Let $C\in\vd_{xy_1}$ be a set in $I$ such that $z\in C$ holds. 
	
	Now, choose a countable basis $\mathcal{B}\in\vd_x$ for the space $Z$. Note that $\mathcal{B}\subset\vd_x$ by Corollary~\ref{wocorollary}. Let $E_0=\{O\in\mathcal{B}\colon z\in O\}$ and $E_1=\{O\in\mathcal{B}\colon O\cap C=0\}$. These are disjoint subsets of $\mathcal{B}$ in $\vd_{xy_0}$ and $\vd_{xy_1}$ respectively. By Proposition~\ref{separationproposition}, there is a $\vd_x$-set $F\subset\mathcal{B}$ such that $E_1\subseteq F$ and $E_0\cap F=0$. Let $C=X\setminus\bigcup F$. This is a closed $\vd_x$ subset of $X$, and by its definition it contains the point $z$ and it is a subset of $C$. The latter statement implies that $Z\setminus\bigcup F\in I$ holds, and the former statement yields a contradiction with the choice of the point $z$.
\end{proof}

\noindent The independence axiom has several interesting mathematical consequences not mentioning definability or $\forkindep$ already at this stage.

\begin{theorem}
	A union of a well-orderable collection of well-orderable sets is itself well-orderable.
\end{theorem}

\begin{proof}
	Let $A$ be such a collection, and let $x$ be a set of ordinals such that $A\in\vd_x$. Corollary~\ref{wocorollary} shows that $A\subset\vd_x$, and then in turn for each element $B\in A$, $B\subset\vd_x$. It follows that $\bigcup A\subset\vd_x$, therefore $A$ is well-orderable.
\end{proof}

\begin{theorem}
	Every linearly ordered set is a union of a well-orderable collection of bounded sets.
\end{theorem}

\begin{proof}
	Let $\langle L, \leq\rangle$ be a linear ordering. Let $x$ be a set of ordinals such that $\langle L, \leq\rangle\in\vd_x$. It will be enough to show that $L$ is the union of all bounded sets in $\vd_x$, since the class $\vd_x$ is canonically well-ordered in ZF. Suppose towards a contradiction that the set $C=L\setminus\bigcup\{A\in\vd_x\colon A\subseteq L$ is bounded$\}$ is nonempty. By the definability axiom, there is a $\vd$-function $F$ and a nonempty set $B\in \vd_x$ of sets of ordinals such that for every $y\in B$, $F(x, y)\in C$ holds.
	
	Use Corollary~\ref{tuplecorollary} to find a triple $\langle y_0, y_1, y_2\rangle$ independent over $x$ such that its entries all belong to the set $B$. Reindexing if necessary, I may assume that $F(x, y_0)\leq F(x, y_1)\leq F(x, y_2)$ holds.
	By the product property, there must be $\vd_x$ sets $B_0, B_1, B_2\subset B$ such that $y_0\in B_0$, $y_1\in B_1$, and $y_2\in B_2$ hold, and for all $z_0\in B_0, z_1\in B_1$, and $z_2\in B_2$, if $z_0\forkindep[x]z_1$ then $F(x, z_0)\leq F(x, z_1)$, and if $z_1\forkindep[x]z_2$ then $F(x, z_1)\leq F(x, z_2)$.
	
	Now, work to show that for every $y'_0\in B_0$, $F(x, y'_0)\leq F(x, y_2)$. To see this, use Proposition~\ref{extensionproposition} to find $y'_1\in B_1$ such that $y'_0y_2\forkindep[x]y'_1$, use the choice of the sets $B_0, B_1$ and $B_2$ to conclude that $F(x, y'_0)\leq F(x, y'_1)$ and $F(x, y'_1)\leq F(x, y_2)$,
	and use the transitivity of the relation $\leq$ to conclude that $F(x, y'_0)\leq F(x, y_2)$.
	
	Finally, by the previous paragraph, the set $\{F(x, y'_0)\colon y'_0\in B_0\}$ is in $\vd_x$ and bounded and contains $F(x, y_0)$, while $F(x, y_0)$ should belong to no such set, reaching a contradiction. 
\end{proof}

\section{The inaccessibility axiom}
\label{inaccessibilitysection}

In many arguments, one needs to assume that the $\hvd_x$ models are fairly thin so that they can be manipulated easily. This is the motivation for the following axiom.

\begin{definition}
\textnormal{(inaccessibility axiom)} There is no $\omega_1$-sequence of reals.
\end{definition}

\noindent To state the instrumental strengthening of the inaccessibility axiom, let $A$ be a countable set and $x$ be a set of ordinals such that $A\in\vd_x$. Then $\power(A)\cap\vd_x$ is countable. To see this, choose a bijection $\pi\colon A\to\gw$ and naturally extend it to a bijection between $\power(A)$ and $\power(\gw)$. Since $\power(A)\cap\vd_x$ is well-orderable by the canonical well-ordering of $\vd_x$, $\pi'' (\power(A)\cap\vd_x)$ is well-orderable as well, and by the inaccessibility axiom it must be countable.

All of the results of this section are proved in ZF+DC plus the definability, independence, and inaccessibility axioms. One attractive consequence is the verification of the result of  T{\" o}rnquist \cite{tornquist:mad} (see also \cite{mathias:happy}) in the Solovay model via a brief forcing-free argument. Recall that a \emph{MAD family} is an infinite subset of $\power(\gw)$ in which any two distinct elements have finite intersection, and it is maximal such with respect to inclusion.

\begin{theorem}
\label{tornquisttheorem}
There are no MAD families.
\end{theorem}

\begin{proof}
Towards a contradiction, assume that $A\subset\power(\gw)$ is a MAD family. Let $x$ be a set of ordinals such that $A\in\vd_x$ holds. Let $I$ be the ideal on $\gw$ generated by $A$. Since $I\cap\hvd_x$ belongs to $\vd_x$, it is an ideal there, and $\hvd_x$ is a model of the axiom of choice, there is in $\hvd_x$ a nonprincipal ultrafilter $U$ on $\gw$ disjoint from $I\cap\hvd_x$. Let $B=\{\langle b, c\rangle\colon b$ diagonalizes $U$ and $c\in A$ is a set such that $b\cap c$ is infinite$\}$. The set $B$ is clearly in $\vd_x$. It is also nonempty: since there is no $\gw_1$-sequence of reals and the model $\hvd_x$ contains a well-ordering of its reals, it must contain only countably many reals, the set $U$ is countable, there are sets diagonalizing it, and each such set has infinite intersection with some element of $A$ by the maximality of $A$. Use the extension axiom to find pairs $b_0c_0\forkindep[x]b_1c_1$ in the set $B$.

Now, argue that $c_0\neq c_1$. If these two sets were equal, by the independence assumption and Proposition~\ref{intersectionproposition} they would belong to $\hvd_x$, therefore to $I\cap\hvd_x$, their common complement would belong to $U$, and the set $b_0$ would be a subset of the complement up to finitely many exceptions. This contradicts the assumption that $b_0\cap c_0$ is infinite.

Since $c_0$ and $c_1$ are distinct elements of $A$, they have finite intersection. By Proposition~\ref{separationproposition}, there is a set $d\subset\gw$ in $\hvd_x$ such that $c_0\subseteq d$ and $c_1\cap d=0$ up to finitely many exceptions. Consider the question whether $d$ belongs to $U$ or not. If the answer is affirmative, then $b_1\subseteq d$ up to finitely many exceptions, contradicting the assumption that $b_1\cap c_1$ is infinite. If the answer is negative, then $b_0\cap d$ is finite, contradicting the assumption that $b_0\cap c_0$ is infinite. Contradiction.
\end{proof}

\begin{theorem}
Let $Z$ be a Polish space and $I$ be a $\gs$-ideal $\gs$-generated by closed sets. Then $I$ is closed under well-ordered unions.
\end{theorem} 

\begin{proof}
Let $A$ be a well-orderable subset of $I$. Let $x$ be a set of ordinals such that $Z$, $I$, and $A$ all belong to $\vd_x$,  Note that $A\subset\vd_x$ holds by Corollary~\ref{wocorollary}. By Corollary~\ref{idealcorollary}, each set in $A$ is a subset of the union of all closed sets in $I$ which belong to $\vd_x$; it follows that $\bigcup A$ is a subset of this union. 

To complete the argument, it will be enough to show that there are only countably many closed subsets of $Z$ in $\vd_x$; this is where the inaccessibility axiom enters the argument. Let $\mathcal{B}$ be a countable basis of $Z$ in $\vd_x$; $\mathcal{B}\subset\vd_x$ follows from Corollary~\ref{wocorollary}.
Each closed set $C$ can be identified with the set of all sets in $\mathcal{B}$ disjoint from $C$, and there are only countably many subsets of $\mathcal{B}$ in $\vd_x$ as $\vd_x$ is well-orderable and $\mathcal{B}$ is countable.
\end{proof}

\noindent The following important consequence was proved by Stern in the Solovay model \cite{stern:borel}. Our axiomatic system allows an entirely transparent proof with no forcing and minimal amount of descriptive set theory.

\begin{theorem}
\label{sterntheorem}
For a countable ordinal $\ga$, there is no injective $\gw_1$-sequence of $\mathbf{\Pi}^0_\ga$-subsets of a given Polish space.
\end{theorem}

\begin{proof}
It is certainly enough to deal with the Baire space $\baire$ only. To start, for an ordinal $\ga$ a \emph{$\ga$-code} is a pair $\langle T, f\rangle$ where $T$ is a well-orderable, well-founded tree of rank $\ga$ and $f$ is a function from terminal nodes of $T$ to $\gwtree$. The function $\eval(T, f, \cdot)\colon T\to\power(\baire)$ is defined by $T$-recursion as $\eval(T, f, t)=\{z\in\baire\colon f(t)\subset z\}$ for terminal node $t\in T$, and $\eval(T, f, t)=\bigcap\{\baire\setminus \eval(T, f, s)\colon s$ is an immediate successor of $t\}$ if $t$ is not terminal. In the end, let $\eval(T, f)=\eval(T, f, t)$ for the largest node $t\in T$.

The proof relies on two central claims which do not use the inaccessibility axiom. Let $x$ be any set of ordinals.

\begin{claim}
Let $\ga$ be any ordinal. Let $B\subset\baire$ be a $\vd_x$-set which has an $\ga$-code. Then $B$ has an $\ga$-code in $\vd_x$.
\end{claim}

\begin{proof}
Note that every well-orderable well-founded tree is isomorphic to some tree of descending sequences of ordinals. Thus, there must be a cardinal $\kappa$ such that $B$ has a $\ga$-code whose tree is a set of decreasing sequences of ordinals smaller than $\kappa$. Let $C$ be the nonempty $\vd_x$-set of all such $\ga$-codes for $B$.

Let $T$ be the set of all pairs $\langle t, D\rangle$ such that $t$ is a decreasing sequence of ordinals smaller than $\kappa$, $D\subseteq C$ is a nonempty $\vd_x$-set, there is an ordinal $\gb=\gb(t, D)\leq\ga$ such that for all codes $\langle U, h\rangle\in D$ it is the case that $t\in U$ and $t$ has rank $\gb$ in $U$, if $\gb=0$ then there is a sequence $f(t, D)\in\gwtree$ such that for all $\langle U, h\rangle\in D$ $h(t)=f(t, D)$ holds, and if $t=0$ then $D=C$. The ordering on $T$ is defined by $\langle t_1, D_1\rangle\leq \langle t_0, D_0\rangle$ if either the two pairs are equal or else $t_0$ is a proper initial segment of $t_1$ and $D_1\subseteq D_0$ holds.

It is not difficult to see that $T$ is a tree, it belongs to $\vd_x$, it is a subset of $\vd_x$ so well-orderable, and it is a well-founded tree of rank $\ga$ as the function $\langle t, D\rangle\mapsto\gb(t, D)$ shows. It will be enough to show that $B=\eval(T, f)$.

To this end, define the function $\eval'\colon T\to\power(\baire)$ by $\eval'(t, D)=\{z\in\baire\colon\forall \langle U, h\rangle\in D\ \langle U, h\rangle\forkindep[x] z\to z\in\eval(U, h, t)\}$.
I will first show that $\eval'(t, D)=\eval(T, f, t, D)$. This is proved by $T$-induction on $\langle t, D\rangle$. The case of nodes $\langle t, D\rangle$ of rank zero follows immediately from the extension property of $\forkindep$. For a node $\langle t, D\rangle$ of nonzero rank, I only need to show that $\eval'(t, D)=\bigcap\{\baire\setminus\eval'(s, E)\colon s$ is a one step extension of $s$ and $E\subseteq D\}$ and then apply the induction hypothesis.

The two inclusions in this equality are both proved in the contrapositive. Suppose that $z\in\baire$ does not belong to the right-hand side, as witnessed by some $\langle s, E\rangle\in T$; so, $z\in\eval'(s, E)$ holds. Use Proposition~\ref{extensionproposition} to find some code $\langle U, h\rangle\in E$ such that $\langle U, h\rangle\forkindep[x]z$. By the definition of $\eval'$, it must be the case that $z\in \eval(U, h, s)$, so $z\notin \eval(U, h, t)$, and the code $\langle U, h\rangle$ witnesses that $z$ does not belong to the left-hand side.

Suppose on the other hand that $z\in\baire$ does not belong to the left-hand side, as witnessed by some code $\langle U, h\rangle\in D$. So, $\langle U, h\rangle\forkindep[x]z$ holds and $z\in \eval(U, h, t)$ fails. Find an immediate successor $s\in U$ of $t$ such that $z\in\eval(U, h, s)$ holds, and write $\gb$ for the rank of $s$ in $U$. By the product property of $\forkindep$, there is a $\vd_x$-set $E\subseteq D$ containing $\langle U, h\rangle$ such that for every $\langle U', h'\rangle\in E$, $s\in U'$ has rank $\gb$, and if $\langle U', h'\rangle\forkindep[x]z$ then $z\in\eval(U', h', s)$ holds. Then the node $\langle s, E\rangle\in T$ witnesses that $z$ does not belong to the right-hand side.

Finally, argue that $\eval(T, h, 0, C)=\eval'(0, C)=B$. For this, suppose that $z\in\baire$ is arbitrary. If $z\in B$, then for every code $\langle U, h\rangle\in C$ whatsoever, $z\in \eval(U, h, 0)$ holds, so $z\in\eval'(0, C)$ follows from the definitions. If, on the other hand, $z\in\eval'(0, C)$, then use Proposition~\ref{extensionproposition} to produce a code $\langle U, h\rangle\in C$ such that $\langle U, h\rangle\forkindep[x]z$. By the definitions, it must be the case that $z\in \eval(U, h, 0)$ holds. Since $\eval(U, h, 0)=B$, $z\in B$ follows.
\end{proof}

\begin{claim}
Let $\ga$ an ordinal, and $\langle T, f\rangle$ a $\ga$-code in $\vd_x$. Then there is a well-ordered $\ga$-code $\langle S, g\rangle\in\hvd_x$ such that $\eval(T, f)=\eval(S, g)$ and $\hvd_x\models |S|\leq\beth_{\ga+1}$.
\end{claim}

\begin{proof}
This is proved by induction on $\ga$. The case of $\ga=1$ is trivial. Now suppose $\ga$ is fixed and the claim has been proved for all smaller ordinals. Let $\langle T, f\rangle$ be a well-orderable $\ga$-code in $\vd_x$. Note that $T\subset\vd_x$ holds by Corollary~\ref{wocorollary}. Let $t\in T$ be the largest element. By the induction hypothesis, for each immediate successor $s$ of $t$ there is a code $\langle S, g\rangle$ such that $\eval(S, g)=\eval(T\restriction s, f)$ of cardinality at most $\beth_\ga$. Let $\langle S_s, g_s\rangle$ be the least such a code in the canonical well-ordering of $\vd_x$. Note that the correspondence $s\mapsto S_s, g_s$ is in $\vd_x$. Now, in $\hvd_x$ there are only $\beth_{\ga+1}$-many options for the codes $\langle S_s, g_s\rangle$ up to isomorphism. Amalgamate all of them into a single tree, appending one node on the top, and the resulting code confirms the induction hypothesis at $\ga$.
\end{proof}

\noindent To prove the theorem, suppose towards a contradiction that $\ga\in\gw_1$ is an ordinal and $\langle B_\gb\colon\gb\in\gw_1\rangle$ is an injective sequence of $\mathbf{\Pi}^0_\ga$ sets. Let $x$ be a set of ordinals such that $\vd_x$ contains the sequence. It is easy to see in ZF+DC that every $\mathbf{\Pi}^0_\ga$ set has a countable $\ga$-code. The conjunction of the two claims then shows that there is an injection of the collection of $\mathbf{\Pi}^0_\ga$ sets in $\vd_x$ and $\ga$-codes in $\hvd_x$ of cardinality $\beth_{\ga+1}$. Now, there are uncountably many $\mathbf{\Pi}^0_\ga$ sets in $\vd_x$ by the contradictory assumption, and there are only countably many $\ga$-codes in $\hvd_x$ of cardinality $\beth_{\ga+1}$ by the inaccessibility axiom.  This is a contradiction.
\end{proof}

\section{The Baire axiom}
\label{bairesection}

The inaccessibility axiom still does not provide sufficient control over the models $\hvd_x$. The following familiar statement provides plenty of additional requisite strength.

\begin{definition}
The \emph{Baire axiom} asserts that every subset of $\baire$ has the property of Baire.
\end{definition}

\noindent It turns out that in the presence of an independence relation, this axiom is much stronger than may appear on the first sight. In particular, it implies that all sets of reals are Lebesgue measurable and all subsets of $[\gw]^{\aleph_0}$ have the Ramsey property. All results in this section use the axiomatic basis of ZF+DC plus definability, independence, inaccessibility, and the Baire axiom. They are easiest to present using the following variation of the forcing relation

\begin{definition}
Let $P$ be a countable poset, and let $x$ be a set of ordinals such that $P\in\vd_x$. 

\begin{enumerate}
	\item A filter $g\subset P$ is \emph{generic over} $\vd_x$ if it has nonempty intersection with every dense subset of $P$ which belongs to $\vd_x$;
	\item if $n\in\gw$ and $\bar\tau\in\vd_x$ is an $n$-tuple of $P$-names and $p\in P$ is a condition and $\phi$ is a formula in the language of set theory with $n$ many free variables, write $p\Vvdash_x\phi(\bar\tau)$ if for every filter $g\subset P$ generic over $\vd_x$ containing $p$, $\phi(\bar\tau/g)$ holds.
\end{enumerate}
\end{definition}

\noindent Several remarks are in order. First of all, Corollary~\ref{wocorollary} shows that as $P$ is well-orderable, $P\subset\vd_x$ holds. By the inaccessibility axiom then, $P$ has only countably many subsets in $\vd_x$, and many filters generic over $\vd_x$ exist. The valuation $\bar\tau/g$ is defined for any filter on $P$ by the usual transfinite recursion formula. The relation $\Vvdash\phi$ belongs to $\vd_x$. Note also the difference between $\Vvdash$ and $\Vdash$: $\Vdash$ speaks about satisfaction in the generic extension of $\hvd_x$, while $\Vvdash$ speaks about satisfaction in the surrounding universe.

\begin{proposition}
	\label{nodependproposition}
The definition of $\Vvdash$ does not depend on $x$.
\end{proposition}

\noindent That is to say that for all sets $x$ of ordinals such that $P, \bar\tau\in\vd_x$, the truth value of $p\Vvdash_x\phi(\bar\tau)$ is the same. 

\begin{proof}
Suppose towards a contradiction that there is $x$ such that $p\Vvdash_{x}\phi(\bar\tau)$ fails, and for some $y$ $p\Vvdash_{y}\phi(\bar\tau)$ holds. Consider the following two nonempty classes definable from $x$ and a finite tuple of ordinal parameters: $B_0=\{g\subset P\colon g$ is generic over $\vd_{x}$, $p\in g$ holds, and $\phi(\bar\tau/g)$ fails$\}$ and $B_1=\{y\colon P,\bar\tau\in\vd_y$ and $p\Vvdash_y\phi(\bar\tau)$ holds$\}$. By Proposition~\ref{extensionproposition}, there must be $g\in B_0$ and $y\in B_1$ such that $g\forkindep[x]y$. It must be the case that $g\subset P$ is generic over $\vd_y$: otherwise, there would be a dense set $C\subset P$ in $\vd_y$ such that $C\cap g=0$, by Proposition~\ref{separationproposition} there would be $D\subset P$ in $\vd_{x}$ such that $C\subset D$ and $g\cap D=0$, and this would contradict the genericity of $g$ over $\vd_x$. Since $p\Vvdash_y\phi(\bar\tau)$, this means that $\phi(\bar\tau/g)$ holds, contradicting the assumption that $g\in B_0$.
\end{proof}

\noindent As a result, below I drop the subscript from the relation $\Vvdash$. The next proposition is in effect the forcing theorem for $\Vvdash$; it is central for the numerous corollaries that follow.

\begin{proposition}
\label{baireproposition}
Suppose that $P$ is a countable poset, $x$ is a set of ordinals such that $P\in\vd_x$, $\bar\tau\in\vd_x$ is an $n$-tuple of $P$-names, $\phi$ is a formula in the language of set theory with $n$ free variables, and $g\subset P$ is a filter generic over $\vd_x$. The following are equivalent:

\begin{enumerate}
	\item $\phi(\bar\tau/g)$ holds;
	\item there is $p\in g$ such that $p\Vvdash\phi(\bar\tau)$.
\end{enumerate}
\end{proposition}

\begin{proof}
	(2) implies (1) by the definition of the $\Vvdash$ relation. To see that (1) implies (2), it is enough to show that the set $D=\{p\in P\colon p\Vvdash\phi(\bar\tau)$ or $p\Vvdash\lnot\phi(\bar\tau)\}$ is a dense subset of $P$; then, as this is a $\vd_x$ set, the filter $g$ must contain one of the conditions in $D$ and (2) follows.

Suppose towards a contradiction that the set $D\subseteq P$ is not dense, and let $p\in P$ be a condition with no element of $D$ below it.  Consider the space $\power(P)$ with the usual topology, and the space $Z\subset\power(P)$ of all subsets $h\subset P$ which are filters, and in addition for every $q\in P$, either $q\in h$ or else $h$ contains an element incompatible with $q$. It is not difficult to see that $Z$ is a $G_\gd$ subset of $\power(P)$, therefore Polish and in fact, if $P$ is atomless then $Z$ is homeomorphic to the Baire space. The topology of $Z$ is generated by sets $O_q=\{h\in Z\colon q\in h\}$.

Now, the Baire axiom shows that the set $A=\{h\in Z\colon\phi(\bar\tau/h)\}$ has the Baire property. Thus, there is a condition $q\leq p$ such that the set $A$ or its complement is comeager in $O_q$. Suppose for definiteness that the former is the case. Then there is a countable collection $\mathcal{D}$ of dense subsets of $P$ such that every filter $h\in Z$ containing $q$ and having nonempty intersection with every set in $\mathcal{D}$ belongs to $A$. Let $y$ be a set of ordinals such that $P, \tau, x, \mathcal{D}\in\vd_y$. The conclusion is that for every filter $h\subset P$ generic over $\vd_y$ and containing the condition $q$, $\phi(\bar\tau/h)$ holds, in other words $q\Vvdash\phi(\bar\tau)$ has been verified (note the use of Proposition~\ref{nodependproposition}). This in turn contradicts the choice of the condition $p$.
\end{proof}

\noindent There are two corollaries which may look technical but in fact, they are very important for the smooth development of balanced forcing.

\begin{corollary}
	For every set $x$ of ordinals, every countable poset $P\in\hvd_x$, and every filter $g\subset P$ generic over $\hvd_x$, $\hvd_x[g]=\hvd_{xg}$ holds.
\end{corollary}

\begin{proof}
	The left-to-right inclusion follows from the definition of a generic extension. For the right-to-left inclusion, it is only necessary to show that every set of ordinals definable from $x$ and $g$ belongs to $\hvd_x[g]$. To this end, let $A$ be a set of ordinals in $\hvd_{xg}$, and let $\phi$ be a formula such that $A=\{\alpha\colon \phi(x, g, \alpha)\}$. Let $\tau\in\hvd_x$ be the canonical $P$-name for the generic filter, and in the model $\hvd_x[g]$ consider the set $B=\{\ga\colon\exists p\in g\ \hvd_x\models p\Vvdash\phi(\check x, \tau, \check\ga)\}$. By Proposition~\ref{baireproposition}, $A=B$ holds, and the proof is complete.
\end{proof}

\begin{corollary}
	\label{genericitycorollary}
	Let $P$ be a countable forcing notion and $x, y$ be sets of ordinals such that $P\in\vd_x$, and $g\subset P$ be a $\vd_x$-generic filter. The following are equivalent:
	
	\begin{enumerate}
		\item $g\subset P$ is a filter generic over $\vd_{xy}$;
		\item $y\forkindep[x]g$.
	\end{enumerate}
\end{corollary}

\noindent Note that as $P$ is well-orderable, $P\subset\vd_x$ holds by Corollary~\ref{wocorollary}, so it makes sense to talk about $\vd_x$-generic filters as those filters on $P$ meeting all $\vd_x$ subsets of $P$ which are dense in $P$.

\begin{proof}
	For (2)$\to$(1), suppose that $y\forkindep[x]g$ holds, let $O\subset P$ is any $\vd_{xz}$-set which is open dense in $P$, and argue that it must have a nonempty intersection with $g$. If this failed, there would be a set $D\in\vd_x$ such that $g\subset D$ and $D\cap O$ is empty by Proposition~\ref{separationproposition}. Such a set must be somewhere dense since it is in $\vd_x$ and contains a $\vd_x$-generic filter as a subset. A somewhere dense set cannot be disjoint from the open dense set $O$. This contradiction shows that $O\cap g\neq 0$ holds as desired. 
	
	If the implication (1)$\to$(2) failed, then there would have to be a condition $p\in g$ such that $p\Vvdash\lnot\check y\forkindep[x] g$, so the set $A=\{z\colon p\Vvdash\lnot\check z\forkindep[x]\tau\}$ is nonempty, where $\tau$ is the canonical $P$-name for its generic filter. Use the extension property of $\forkindep$ to find $z\in A$ such that $z\forkindep[x]g$. The previously proved implication (2)$\to$(1) applied to $z$ and $g$ shows that $g$ is $\vd_{xz}$-generic filter on $P$, and this contradicts the choice of $p$ and $z$.
\end{proof}

\noindent To conclude this section, I prove two theorems which show that in the presence of the independence axiom, the Baire property of sets of reals has multiple consequences which in its basence are well-known not to follow. The first is the main point of the ground-breaking work of Solovay \cite{solovay:model}.

\begin{theorem}
\label{solovaytheorem}
Let $\mu$ be the usual Borel probability measure on $\cantor$. Every subset of $\cantor$ is $\mu$-measurable.
\end{theorem}

\begin{proof}
Let $A\subset\cantor$ be any set. Let $x$ be a set of ordinals such that $A\in\vd_x$, and let $\phi$ be a formula with parameter $x$ such that $A=\{z\in\cantor\colon\phi(z, x)\}$. Let $P$ be the poset of Borel $\mu$-positive subsets of $\cantor$ as evaluated in the model $\hvd_x$. Each $p\in P$ is then naturally identified with a Borel subset of $\cantor$ in the universe (as opposed to $\hvd_x$, for the theory of interpretations of Borel sets between models of set theory see \cite{z:interpretations}). Let $\tau$ be the $P$-name for the canonical generic element of $Z$ (the ``random real''). Let $B=\{p\in P\colon p\Vvdash\phi(\tau, \check x)\}$. Note that $B$ is a countable collection of Borel sets. It will be enough to show that $A=\bigcup B$ modulo a null set of exceptions.

Let $C\subset\cantor$ be the union of all $\mu$-null $G_\gd$-subsets of $\cantor$ coded in $\hvd_x$. By the inaccessibility axiom, this is a union of a countable family of $\mu$-null sets, therefore $\mu$-null. It will be enough to show that if $z\in\cantor\setminus C$ then $z\in A\liff z\in B$ holds. Let $z\in\cantor\setminus C$ be a point, and let $g=\{p\in P\colon z\in p\}$. By a theorem of Solovay \cite[Lemma 26.4]{jech:newset}, the set $g\subset P$ is a filter generic over $\hvd_x$; in addition, $z=\tau/g$. Now, if $z\in B$ then $z\in A$ holds by the definition of the set $B$. On the other hand, if $z\in A$ then $g$ must meet a condition $p$ such that $p\Vvdash\phi(\tau, \check x)$ by Proposition~\ref{baireproposition}, so $z\in B$.
\end{proof}

\noindent The following corollary verifies a feature of the Solovay model proved in \cite{mathias:happy}.

\begin{theorem}
\label{mathiastheorem}
For every set $A\subset [\gw]^{\aleph_0}$ there is an infinite set $b\subset\gw$ such that $[b]^{\aleph_0}$ is either a subset of $A$ or disjoint from $A$.
\end{theorem}

\begin{proof}
Let $x$ be a set of ordinals such that $A\in\vd_x$ holds, and let $\phi$ be a formula with parameter $x$ such that $A=\{z\in [\gw]^{\aleph_0}\colon\phi(z, x)\}$. Write $P$ for the Mathias forcing as evaluated in the model $\hvd_x$. $P$ is countable by the inaccessibility axiom. Let $\tau$ be the canonical $P$-name for its generic subset of $\gw$.

Consider the set $\gs$ which is the union of $\{\langle p, 0\rangle\colon p\Vvdash\phi(\tau, \check x)\}$ together with the set $\{\langle p, 1\rangle\colon p\Vvdash\lnot\phi(\tau, \check x)\}$. Clearly, $\gs\in\hvd_x$ holds, and Proposition~\ref{baireproposition} implies that $\gs$ is a name for an element of $2$. Working in the model $\hvd_x$, by a result of Mathias \cite[Lemma 26.34]{jech:newset} there is an infinite set $c\subset\gw$ in $\vd_x$ such that the condition $\langle 0, c\rangle\in P$ decides the value of $\gs$. Suppose for definiteness that this condition forces $\gs=\check 0$. Step out of the model $\hvd_x$ and let $b\subset c$ be a Mathias generic real. I claim that $[b]^{\aleph_0}\subset A$.

Indeed, if $d\subset b$ is an infinite set, then again by a result of Mathias \cite[Corollary 26.38]{jech:newset} $d$ is a $P$-generic real meeting the condition $\langle 0, c\rangle$, so its associated generic filter contains a condition $p\in P$ such that $\langle p, \check 0\rangle\in\gs$. By the definition of the name $\gs$, it follows that $d\in A$ as required.
\end{proof}

\section{The Gandy--Harrington axiom}
\label{ghasection}

The previous axioms do not seem to address generalizations of any dichotomies such as the Silver dichotomy to all sets, even though suitable generalization are well-known to hold in the Solovay model \cite{kanovei:ulmsolovay, kanovei:boundedness}. There is a simple forcing-free axiom which holds in the Solovay model and seems to imply them all. It is motivated by the proofs of effective descriptive set theory as presented in \cite{kanovei:book}. 

\begin{definition}
Let $x$ a set of ordinals and $A\in\vd_x$ be a nonempty set. The \emph{Gandy--Harrington game} on $x$ and $A$ is played between Players I and II, who alternately play nonempty $\vd_x$-sets $A_n\subset A_0$ such that $A=A_0\supseteq A_1\supseteq \dots$. Player II wins if the intersection $\bigcap_nA_n$ is nonempty.
\end{definition}

\begin{definition}
The \emph{Gandy--Harrington axiom} is the statement that for every set $x$ of ordinals and every nonempty $\vd_x$ set $A$, Player II has a winning strategy in the Gandy--Harrington game on $x$ and $A$.
\end{definition}

\noindent The results below are proved in the theory ZF+DC plus the definability, independence, and Gandy--Harrington axioms. The take-home message is that these are pure fusion-type proofs, with no forcing and no reflection theorems appearing anywhere in them. The first result is a variation on the Silver dichotomy \cite{silver:dichotomy}, proved in the Solovay model by Stern \cite{stern:borel}.

\begin{theorem}
\label{silvertheorem}
Let $E$ be an equivalence relation on $\cantor$. Either the set of all $E$-equivalence classes is well-orderable, or there is a perfect set of $E$-inequivalent elements.
\end{theorem}

\noindent The theorem is a corollary to a more technical general proposition.

\begin{proposition}
\label{perfectproposition}
Let $x$ be a set of ordinals and $A\subset\cantor$ be a $\vd_x$-set of reals. Either $A$ is well-orderable, or $A$ contains a perfect subset of elements which are in finite tuples independent over $x$.
\end{proposition}

\begin{proof}
If $A\subset\vd_x$ holds, then the first alternative occurs. If $A\not\subset\vd_x$, then replacing $A$ with $A\setminus\vd_x$ if necessary, I may assume that $A$ contains no $\vd_x$ elements. By recursion on $n\in\gw$ build

\begin{itemize}
\item nonempty $\vd_x$ subsets $A_t\subseteq A$ for $t\in 2^n$ such that all elements of $A_t$ share the same initial segment of length $n$;
\item for each nonzero $m\in\gw$, each $k\leq n$ and each injective $m$-tuple $\bar u\in 2^{k}$ a nonempty $\vd_x$-set $B_{\bar u}\subset\{\bar y\in A^m\colon\bar y$ is independent over $x$ and $\forall i\in m\ \bar u(i)\subset\bar y(i)\}$ and a winning strategy $\gs_{\bar u}$ for Player II in the Gandy--Harrington game associated with $x$ and $B_u$;
\item for each nonzero $m\in\gw$, each $k\leq n$, each injective $m$-tuple $\bar u\in 2^{k}$, and each $m$-tuple $\bar v\in 2^n$ such that $\forall i\in m\ \bar u(i)\subseteq\bar v(i)$ a play $p(\bar u, \bar v)$ against the strategy $\gs_{\bar u}$ of length at least $n+1-k$ in which the strategy $\gs_{\bar u}$ makes the last move which is a superset of the set $\{\bar y\in\prod_{t\in\bar v}A_t\colon \bar y$ is independent over $x\}$;
\end{itemize}

\noindent and do it all in such a way that if $k\leq n_0<n_1$ and $\bar u\in 2^{k}$ and $\bar v_0\in 2^{n_0}$ and $\bar v_1\in 2^{n_1}$ are injective $m$-tuples and $\forall i\in m\ \bar u(i)\subseteq\bar v_0(i)\subseteq\bar v_1(i)$, then $p(\bar u, \bar v_0)\subseteq p(\bar u, \bar v_1)$.

The recursion step is performed by one-by-one extending the plays $p(\bar u,\bar v)$ and applying Corollary~\ref{duplecorollary} repeatedly. The straightforward bookkeeping details are left to the reader.

During the recursion on $n$, there are some choices to be made, namely the choice of the strategies $\gs_{\bar u}$. The axiom of dependent choices shows that the recursion can be performed. Once this is done, for each $t\in\cantor$ consider the play $p_t=\bigcup_np_{\{0\}, t_n}$ of the Gandy--Harrington game with $x$ and $A$ against the strategy $\gs_{\{0\} }$. The intersection of the moves played along the game must be nonempty, and it is a singleton by the first item of the recursion on $n$. Let $f\colon\cantor\to\cantor$ be the function assigning to each $t\in\cantor$ the unique point in this singleton set; the function $f$ is continuous. If $\bar y\in\cantor$ is an injective $m$-tuple, find some $k$ such that the $m$-tuple of initial segments of entries on $\bar y$ of length $k$ is injective. Let $p$ be the union of all plays $p(\bar u, \bar v_n)$ where $\bar v_n$ is the tuple of all initial segments of length $n$ of entries on $\bar y$.  The intersection of the sets played along $p$ must be nonempty, and there is only one candidate for its element, namely the tuple $\langle f(\bar y(i))\colon i\in m\rangle$. This tuple must be independent over $x$ by the choice of $B_{\bar u}$ in the second recursion demand above. Thus, the range of $f$ is a perfect subset of $A$ consisting of elements in finite tuples independent over $x$.
\end{proof}

\begin{proof}
To prove Theorem~\ref{silvertheorem}, suppose that $E$ is an equivalence relation on $\cantor$. Let $x$ be a set of ordinals such that $E\in\vd_x$. If all equivalence classes of $E$ are $\vd_x$, then the set of all $E$-classes is well-orderable. Otherwise, one can consider the nonempty $\vd_x$-set $A=\{z\in\cantor\colon [z]_E\notin\vd_x\}$. Clearly, no element of $A$ is in $\vd_x$, so $A$ is not well-orderable by Corollary~\ref{wocorollary}. By Proposition~\ref{perfectproposition}, there is a perfect set $C\subset A$ consisting of elements pairwise independent over $x$. It turns out that $C$ consists of pairwise $E$-inequivalent elements: if $y_0, y_1\in\cantor$ are distinct elements of $C$, then $y_0\forkindep[x]y_1$, so $\vd_{xy_0}\cap\vd_{xy_1}=\vd_x$ by Proposition~\ref{intersectionproposition}. So, if $y_0\mathrel Ey_1$ were the case, their common equivalence class would belong to $\vd_x$, contradicting the choice of the set $A$.
\end{proof}

\noindent The second application uses the graph $G_0$ on $\cantor$ and generalizes its dichotomy as isolated by Kechris, Solecki, and Todorcevic \cite{kechris:chromatic}. Let $u_n\in 2^n$ be a string for each $n\in\gw$ selected in such a way that the collection $\{u_n\colon n\in\gw\}$ is dense in $2^{<\gw}$. Let $G_0$ be the graph on $2^\gw$ connecting two sequences $x, y$ if they differ in exactly one entry, and there is $n\in\gw$ such that $u_n$ is exactly the longest common initial segment of $x$ and $y$.

\begin{theorem}
\label{g0theorem}
Let $G$ be a graph on $\baire$. Either $\baire$ can be covered by a well-ordered set of $G$-anticliques, or there is a continuous homomorphism of $G_0$ to $G$.
\end{theorem}

\begin{proof}
Let $x$ be a set of ordinals such that $G\in\vd_x$. Either $\baire$ is covered by $G$-anticliques which are $\vd_x$, in which case it is covered by a well-ordered set of $G$-anticliques. If this does not occur, let $A=\baire\setminus\bigcup\{B\subset\baire\colon B\in\vd_x$ is a $G$-anticlique$\}$. The set $A$ is $\vd_x$ and nonempty; I will produce a continuous homomorphism of $G_0$ to $G$ whose range is a subset of $A$. 

Pick a winning strategy $\gs$ for Player II in the Gandy--Harrington game associated with $x$ and $A$. For each number $n\in\gw$, let $H_n$ be the graph on $2^n$ connecting strings $s$ and $t$ if there is $i\in n$ such that $u_i$ is the longest common initial segment of $s$ and $t$, and $s$ and $t$ differ only at $i$-th entry. By recursion on $n\in\gw$ build

\begin{itemize}
\item nonempty $\vd_x$ sets $A_n$ consisting of maps from $2^n$ to $\baire$ which are homomorphisms from $H_n$ to $G$. In addition, the projection of $A_n$ into any coordinate $t\in 2^n$ consists of points in $\baire$ which share the same initial segment of length $n$;
\item for each $t\in 2^n$ build plays $p_t$ against the strategy $\gs$ of length at least $n$ in which strategy $\gs$-plays the last move, which is a superset of the projection of $A_n$ to coordinate $t$;
\item for each $i$ build a nonempty $\vd_x$-set $B_i\subset\{\langle y, z\rangle\in A\times A\colon y\mathrel Gz\}$ and a winning strategy $\gs_i$ for Player II in the Gandy-Harrington game associated with $x$ and $B_i$;
\item for every pair $s\mathrel H_nt$ of distinct elements of $2^n$ with longest common initial segment $u_i$ listed so that $u^\smallfrown 0\subseteq s$ and $u^\smallfrown 1\subseteq t$ build a play $p_{s, t}$ against the strategy $\gs_u$ of length at least $n-|u|$ in which the strategy $\gs_u$ makes the last move which is a superset of the projection of $A_n$ into the coordinates $s$ and $t$.
\end{itemize}

\noindent To set up the recursion, let $p_0$ be the two-move counterplay of the Gandy--Harrington game associated with $x$ and $A$ against $\gs$ ending with the strategy playing the set $A_0$. The recursion step is performed in three stages. In the first stage, extend each of the plays $p_t$ by one step and shrink the set $A_n$ to maintain the third demand above. In the second stage, for each $H_n$-connected strings $s, t\in 2^n$ extend the plays $p_{s, t}$ by one step and shrink the set $A_n$ again to maintain the fourth demand above. The result will be a nonempty $\vd_x$ set $A'_n\subseteq A_n$, and plays $p'_t$ and $p'_{st}$.

In the end, note that the projection of $A'_n$ into the $u_n$ coordinate is a nonempty $\vd_x$-set, a subset of $A$. As such, it is not a $G$-anticlique. Pick a winning strategy $\gs_u$ for Player II in the game associated with $x$ and the nonempty set $G\cap A'_n\times A'_n$, and let $p_{u_n}$ be the two move play of this game in which the strategy makes the last move $B$. For each map $f\colon 2^{n+1}\to\baire$ define $f_0\colon 2^n\to\baire$ by $f_0(t)=f(t^\smallfrown 0)$ and $f_1\colon 2^n\to\baire$ by $f_1(t)=f(t\smallfrown 1)$. Now let $A_{n+1}$ be the set of all functions $f\colon 2^{n+1}\to\baire$ such that $f_0\in A'_n$, $f_1\\in A'_n$, and $\langle f(u_n^\smallfrown 0), f(u_n^\smallfrown 1)\rangle\in B$. By the construction, this is a nonempty $\vd_x$-set consisting of homomorphisms from $H_{n+1}$ to $G$. For each $t\in 2^{n+1}$ let $p_{t^\smallfrown 0}=p_{t^\smallfrown 1}=p'_t$, and for each pair $s\mathrel H_nt$ let $p_{s^\smallfrown 0, t^\smallfrown 0}=p_{s^\smallfrown 1, t^\smallfrown 1}=p'_{s, t}$. It is not difficult to see that the recursion demands are satisfied.

The axiom of dependent choices shows that the recursion can be performed. Once this is done, for each $t\in\cantor$ consider the play $p_t=\bigcup_np_{t_n}$. The intersection of the moves played along the game must be nonempty, and it is a singleton by the first item of the recursion on $n$. Let $f\colon\cantor\to\cantor$ be the function assigning to each $t\in\cantor$ the unique point in this singleton set; the function $f$ is continuous. If $s, t\in\cantor$ are distinct and for some $i\in\gw$, $u_i$ is their longest common initial segment and $s$ and $t$ differ only at $i$-th entry, consider the play $\bigcup_np_{s\restriction n, t\restriction n}$ against the strategy $\gs_i$. By the last item of the recursion on $n$, the intersection of the sets played along $p_{s, t}$ must be nonempty, and there is only one candidate for its element, namely the pair $\langle f(s), f(t)\rangle$. It follows that $f(s)\mathrel Gf(t)$. Thus, $f$ is a continuous homomorphism of $G_0$ to $H$ as desired.
\end{proof}

\section{Models for the axioms}
\label{modelsection}

In this section, I will show that one version of the Solovay model satisfies all of the axioms mentioned in this paper. There is nothing novel in these proofs.

Let $V$ be a transitive model of ZFC containing all ordinals, let $\kappa$ be a cardinal strongly inaccessible in $V$, and $G\subset\coll(\gw, <\kappa)$ be a filter generic over $V$. In the model $V[G]$, form the model $W=V(\mathbb{R})$ which is the smallest transitive model of ZF containing $V$ and $\power(\gw)$ as subsets. In order for the model $W$ to satisfy the axioms stated above, one has to make some assumptions on the ground model $V$ to make sure that in all forcing extensions, $V\subseteq\hvd$ holds. One possibility is to assume that $V$ is the constructible universe; another is for $V$ to have rigid cardinal arithmetic, such as

\begin{enumerate}
\item[(*)] for every ordinal $\gb$ and every set $x\subset\gb$, there is a proper class of ordinals $\gg$ such that for all $\ga\in\gb$, the powerset of $\aleph_{\gg+\ga\cdot\gw+5}$ has cardinality $\aleph_{\gg+\ga\cdot\gw+1}$ if $\ga\in A$, and $\aleph_{\gg+\ga\cdot\gw+7}$ if $\ga\notin A$.
\end{enumerate}

\noindent Another approach is to allow parameters in $V$ in the definition of $\hvd$ in the axiomatization. This change does not affect any proofs in the previous sections in any way. There are still other options. For definiteness, I will assume that the ground model $V$ satisfies (*).

For every formula $\phi$, consider the formula $\hat\phi$ with an additional parameter $\kappa$ defined in the following way: $\hat\phi(\bar v)$ states $\coll(\gw, <\kappa)\Vdash V(\mathbb{R})\models\phi(\bar v)$. In this formula, $V(\mathbb{R})$ stands for the smallest transitive model of ZF containing the ground model and $\power(\gw)$ of the $\coll(\gw, <\kappa)$-extension as subsets. In addition, each item $u$ on the variable list $\bar v$ to the right of the forcing relation should be equipped with a check, like so: $\check u$. It denotes the canonical name for $u$ as the element of the ground model. The following well-known fact appears inside the usual proofs of Solovay's measurability theorem \cite[Theorem 8.23]{schindler:set}, \cite[Lemma 26.17]{jech:newset}.

\begin{fact}
\label{solovayfact}
Let $y$ be a set of ordinals in $W$. Then

\begin{enumerate}
\item The model $V[y]$ is a generic extension of $V$ by a poset of cardinality smaller than $\kappa$;
\item for every formula $\phi$ with parameters in $V[y]$, $W\models\phi$ iff $V[y]\models\hat\phi$.
\end{enumerate}
\end{fact}

It is part of the classical work of Solovay that ZF+DC holds in $W$. For the definability axiom, note that as $V$ satisfies the axiom of choice, the $\in$-relation on any transitive set in it is isomorphic to a binary relation on an ordinal in it. Therefore, $W$ internally sees itself as the smallest model of ZF containing all sets of ordinals. A standard constructibility hierarchy argument proves in ZF that in $L(\power(Ord))$, there is a definable surjection from $\power(Ord)\times Ord^{<\gw}$ onto $L(\power(Ord))$. This argument applied in $W$ shows that every set is definable from a set of ordinals (and additional ordinal parameters) there.

For the independence axiom, let $y_0\forkindep[x]y_1$ if $V[x, y_0]$ and $V[x, y_1]$ are mutually generic extensions of $V[x]$. The requisite properties of this relation must be checked one by one. First of all, the relation $\forkindep$ is definable without parameters. For this, it is enough to show that $\hvd^W=V$. The right-to-left inclusion follows from the (*) assumption above. For the left-to-right inclusion, suppose that $\phi(\bar v)$ is a formula with ordinal parameters and $\gb$ is an ordinal. It will be enough to show that the set $\{\ga\in\gb\colon W\models\phi(\ga, \bar v)\}$ (a typical inhabitant of $\hvd^W$) belongs to $V$. Indeed, the set is defined in $V$ as $\{\ga\in\gb\colon V\models\hat\phi(\ga, \bar v)\}$. 

Now, mutual genericity is clearly symmetric. Monotonicity and the transitivity property are easiest to verify by a mutual genericity criterion developed in \cite[Proposition 1.7.9]{z:geometric}: $V[x, y_0]$ and $V[x, y_1]$ are mutually generic if and only if any disjoint sets $a_0\in V[x, y_0]$ and $a_1\in V[x, y_1]$ can be separated by a set of ordinals in $V[x]$. For the monotonicity, if $y'_0$ is a set of ordinals definable from $x$ and $y_0$ as $y'_0=\{\ga\colon W\models\phi(\ga, x, y_0, \bar\gb)\}$ for some formula $\phi$ and a finite tuple $\bar\gb$ of ordinals, then $y'_0\in V[x, y_0]$ holds, as $y'_0=\{\ga\colon V[x, y]\models\hat\phi(\ga, x, y_0, \bar\gb)\}$ by Fact~\ref{solovayfact}. Thus, $V[x, y'_0]\subseteq V[x, y_0]$, and $V[x, z]$ and $V[x, y_1]$ are mutually generic as the above criterion is clearly hereditary to submodels. For the transitivity property, suppose that $x, y_0, y_1, z$ are sets of ordinals such that  $V[x, y_0]$ and $V[x, y_1]$ are mutually generic over $V[x]$ and $V[x, y_0, y_1]$ and $V[x, y_0, z]$ are mutually generic extensions of $V[x, y_0]$. Then, the mutual genericity of $V[x, y_0, z]$ and $V[x, y_1]$ is verified by the above criterion. Suppose that $a_0\in V[x, y_0, z]$ and $a_1\in V[x, y_1]$ are disjoint sets of ordinals. By the mutual genericity of $V[x, y_0, z]$ and $V[x, y_0, y_1]$ there is a set $b\in V[x, y_0]$ such that $a_0\subset b$ and $a_1\cap b=0$. By the mutual genericity of $V[x, y_0]$ and $V[x, y_1]$, there is a set $c\in V[x,]$ of ordinals such that $b\subseteq c$ and $a_1\cap c=0$. Then the set $c\in V[x]$ separates $a_0$ from $a_1$ as desired.

For the extension property, suppose that $x$ is a set of ordinals in $W$, and $A\in W$ is a nonempty set of sets of ordinals in $\vd_x$. Let $\phi$ be a formula with one free variable and parameter $x$ and some other parameters in $V$ defining $A$ in $W$. Let $y\in A$ be an arbitrary set. Since every set of ordinals is generic over $V$ by a poset of cardinality $<\kappa$, there must be a poset $P\in V[x]$ of cardinality $<\kappa$ and a $P$-name $\tau\in V[x]$ and a filter $g\subset P$ in $W$ such that $g$ is generic over $V[x]$ and $\tau/g=A$. By the forcing theorem, there must be a condition $p\in P$ such that $V[x]\models p\Vdash\hat\phi(\check x, \tau)$. In the model $W$, pick filters $g_0, g_1\in P$ mutually generic over $V$, both containing the condition $p$. Let $y_0=\tau/g_0$ and $y_1=\tau/g_1$; these are sets in the set $A$ such that $V[x][y_0]$ and $V[x][y_1]$ are mutually generic over $V[x]$ as desired.

For the product property, suppose that $x$ is a set of ordinals in $W$, $\beta$ is an ordinal, $A\subset\power(\gb)\times\power(\gb)$ is a $\vd_x$-set, and $\langle y_0, y_1\rangle\in A$ is a pair mutually generic over $V[x]$. Let $\phi$ be a formula with parameter $x$ defining the set $A$. This means that there are posets $P_0, P_1$ in $V[x]$ of cardinality $<\kappa$ and names $\tau_0, \tau_1$ such that $P_0\times P_1\Vdash\hat\phi(\tau_0, \tau_1, \check x)$. Let $B_0=\{z_0\colon\exists g_0\subset P_0\ g_0$ generic over $V[x]$ and $z_0=\tau_0/g_0\}$ and similarly for subscript $1$. It is not difficult to see that the sets $B_0$ and $B_1$ work.

The inaccessibility and Baire axioms are part of the classical Solovay's work. For the Gandy--Harrington axiom, first note that by the definability axiom, for every set $x$ of ordinals, every nonempty $\vd_x$-set has a nonempty subset which is an image of an $\vd_x$ function whose domain is an $\vd_x$-set of ordinals. To see this, let $A\in\vd_x$ be a nonempty set definable by some formula $\phi$: $A=\{a\colon\phi(a, x)\}$. Pick an element $a\in A$, a set $y$ of ordinals smaller than some ordinal $\gb$ and a formula $\psi$ defining $a$ from $y$. Define $F$ to be the set of all those pairs $\langle z, b\rangle$ such that $z\subset\gb$ is a set, $\psi$ defines $b$ from $z$, and $\phi(b, x)$ holds. $F$ is a nonempty $\vd_x$ function whose domain is a $\vd_x$-set of ordinals and range is a subset of $A$.

It follows that it is enough to verify the Gandy--Harrington axiom for sets of sets of ordinals. Let $x\subset\gb$ be a set of ordinals, and $A$ a nonempty $\vd_x$-set of sets of ordinals, $A=\{y\subset\ga\colon\phi(y, x)\}$. This is the first move of Player I in the Gandy--Harrington game, $A=A_1$. Working in $V[x]$, find a poset $P$ of cardinality $<\kappa$ and a $P$-name $\tau$ such that $V[x]\models P\Vdash\hat\phi(\tau,\check x)$. To proceed, Player II will need an enumeration $\{D_i\colon i\in\gw\}$ of all dense open subsets of $P$ in the model $V[x]$ (the enumeration will itself be an element of the Solovay model surrounding $V[x]$).  As the play proceeds, Player II will produce a descending sequence $\langle p_i\colon i\in\gw\rangle$ of conditions such that $p_{i+1}\in D_i$ and play sets $A_{2i+1}$ in such a way that $A_{2i+1}=\{\tau/g\colon g\subset P$ is a filter generic over $V[x]$ containing the condition $p_i$. In the first round, $p_0\in P$ is arbitrary. Suppose that the condition $p_i$ has been found and the set $A_{2i+1}$ has been played. Player I answers with a nonempty subset $A_{2i+2}$. 

I claim that there is a condition $q\leq p_i$ in $D_i$ such that $\{\tau/g\colon g\subset P$ is a filter generic over $V[x]$ and $q\in g\}\subset A_{2i+2}\}$. To see this, pick any filter $g\subset P$ generic over $V[x]$, with $p_i\in g$, such that $\tau/g\in A_{2i+1}$. As this is a $\vd_x$-set, there must be a formula $\psi$ such that $A_{2i+2}=\{y\subset\ga\colon\psi(y, x)\}$, and by the forcing theorem and Fact~\ref{solovayfact}, there must be a condition in the filter $g$ which forces $\hat\psi(\tau, \check x)$. Then any stronger condition which is in $D_{i+1}$ and below $p_i$ will work. Player II will select the first condition satisfying the recursion demands in the previous paragraph in some fixed well-ordering of $V[x]$. This completes the description of the strategy. In the end, Player II must have won: writing $g$ for the filter on $P$ generated by his sequence of conditions, $\tau/g\in\bigcap_iA_i$ holds.

As a final remark, I prove a Kim-Pillay style theorem regarding the $\forkindep$ relation in the Solovay model.

\begin{proposition}
	\label{kimpillayproposition}
	In the Solovay model, if $\forkindep$ is any relation satisfying the independence axiom, then it is the relation $y_0\forkindep_xy_1$ iff $V[x][y_0]$, $V[x][y_1]$ are mutually generic extensions of $V[x]$.
\end{proposition}

\begin{proof}
	Note that the initial assumptions on the ground model used in the construction of the Solovay model imply that $V[x]=\hvd_x$ for any set $x$ of ordinals in it. Now, let $x, y_0, y_1$ be any sets of ordinals. Fact~\ref{solovayfact}(1) shows that $V[x][y_0]$ and $V[x][y_1]$ are both generic extensions of $V[x]$ via a poset of cardinality smaller than $\kappa$; i.e. a poset which is countable in the Solovay model. Corollary~\ref{genericitycorollary} then shows that $y_0\forkindep[x]y_1$ is equivalent to $V[x][y_0], V[x][y_1]$ being mutually generic extensions of $V[x]$ as desired.
\end{proof}

\begin{question}
	Can the uniqueness of $\forkindep$ be proved just from the axiomatization?
\end{question}

\bibliographystyle{plain} 
\bibliography{odkazy,zapletal,shelah}

\def\germ{\frak} \def\scr{\cal} \ifx\documentclass\undefinedcs
  \def\bf{\fam\bffam\tenbf}\def\rm{\fam0\tenrm}\fi % f**k-amstex!
  \def\defaultdefine#1#2{\expandafter\ifx\csname#1\endcsname\relax
  \expandafter\def\csname#1\endcsname{#2}\fi} \defaultdefine{Bbb}{\bf}
  \defaultdefine{frak}{\bf} \defaultdefine{=}{\B} % doublef**k-amstex!!
  \defaultdefine{mathfrak}{\frak} \defaultdefine{mathbb}{\bf}
  \defaultdefine{mathcal}{\cal}
  \defaultdefine{beth}{BETH}\defaultdefine{cal}{\bf} \def\bbfI{{\Bbb I}}
  \def\mbox{\hbox} \def\text{\hbox} \def\om{\omega} \def\Cal#1{{\bf #1}}
  \def\pcf{pcf} \defaultdefine{cf}{cf} \defaultdefine{reals}{{\Bbb R}}
  \defaultdefine{real}{{\Bbb R}} \def\restriction{{|}} \def\club{CLUB}
  \def\w{\omega} \def\exist{\exists} \def\se{{\germ se}} \def\bb{{\bf b}}
  \def\equivalence{\equiv} \let\lt< \let\gt>
\begin{thebibliography}{10}

\bibitem{jech:newset}
Thomas Jech.
\newblock {\em Set Theory}.
\newblock Springer Verlag, New York, 2002.

\bibitem{kanovei:ulmsolovay}
Vladimir Kanovei.
\newblock An {U}lm-type classification theorem for equivalence relations in
  {S}olovay model.
\newblock {\em Journal of Symbolic Logic}, 62:1333--1351, 1997.

\bibitem{kanovei:book}
Vladimir Kanovei.
\newblock {\em Borel Equivalence Relations}.
\newblock University Lecture Series 44. American Mathematical Society,
  Providence, RI, 2008.

\bibitem{kanovei:boundedness}
Vladimir Kanovei and V.~A. Lyubetsky.
\newblock On effective $\sigma$-boundedness and $\sigma$-compactness in
  {S}olovay's model.
\newblock {\em Mathematical Notes}, 98:273--282, 2015.

\bibitem{kechris:chromatic}
Alexander~S. Kechris, Slawomir Solecki, and Stevo Todorcevic.
\newblock Borel chromatic numbers.
\newblock {\em Advances in Mathematics}, 141:1--44, 1999.

\bibitem{pillay:simple}
Byunghan Kim and Anand Pillay.
\newblock Simple theories.
\newblock {\em Annals of Pure and Applied Logic}, 88:149--164, 1997.

\bibitem{z:geometric}
Paul Larson and Jind{\v r}ich Zapletal.
\newblock {\em Geometric set theory}.
\newblock AMS Surveys and Monographs. American Mathematical Society,
  Providence, 2020.

\bibitem{marker:book}
David Marker.
\newblock {\em Model theory: An introduction}.
\newblock Graduate Texts in Mathematics 217. Springer Verlag, 2002.

\bibitem{mathias:happy}
Adrian~R.D. Mathias.
\newblock Happy families.
\newblock {\em Annals of Mathematical Logic}, 12:59--111, 1977.

\bibitem{schindler:set}
Ralf Schindler.
\newblock {\em Set Theory: Exploring Independence and Truth}.
\newblock Universitext. Springer Verlag, New York, 2014.

\bibitem{silver:dichotomy}
Jack Silver.
\newblock Counting the number of equivalence classes of {B}orel and coanalytic
  equivalence relations.
\newblock {\em Annals of Mathematics}, 18:1--28, 1980.

\bibitem{solovay:model}
Robert~Martin Solovay.
\newblock A model of set theory in which every set of reals is lebesgue
  measurable.
\newblock {\em Annals of Mathematics}, 92:1--56, 1970.

\bibitem{stern:borel}
Jacques Stern.
\newblock On {L}usin's restricted continuum problem.
\newblock {\em Annals of Mathematics}, 120:7--37, 1984.

\bibitem{tornquist:mad}
Asger T{\" o}rnquist.
\newblock Definability and almost disjoint families.
\newblock {\em Advances in Mathematics}, 330:61--73, 2018.

\bibitem{z:interpretations}
Jind{\v r}ich Zapletal.
\newblock Interpreter for topologists.
\newblock {\em Journal of Logic and Analysis}, 7:1--61, 2015.

\end{thebibliography}

\end{document}